\documentclass[11pt,twoside, a4paper, english, reqno]{amsart}
\usepackage{amsmath, graphicx, varioref, amscd, amssymb,color, bm, amsthm, epsfig, color, enumerate}
\usepackage{fancybox}
\usepackage{fancybox}
\usepackage{bbm}
\usepackage{esint}
\usepackage[pdftex, colorlinks=true]{hyperref}
\usepackage{upref}
\usepackage{xcolor}

\numberwithin{equation}{section}
\allowdisplaybreaks

\newtheorem{theorem}{Theorem}[section]
\newtheorem{lemma}[theorem]{Lemma}

\theoremstyle{definition}
\newtheorem{definition}[theorem]{Definition}

\theoremstyle{remark}

\newcommand{\Div}{\operatorname{div}}

\newcommand{\vq}{\bm{q}}
\newcommand{\vu}{\bm{u}}

\newcommand{\dx}{{\rm d}x}
\newcommand{\dt}{{\rm d}t}
\newcommand{\dxdt}{\dx \ \dt}
\newcommand{\dq}{{\rm d}\vq}

\newcommand{\mus}{\mu^S}
\newcommand{\mub}{\mu^B}
\newcommand{\muns}{\mu_n^S}
\newcommand{\munb}{\mu_n^B}

\newcommand{\vc}[1]{{\bm{#1}}}

\newcommand{\R}{\mathbb{R}}
\newcommand{\N}{\mathbb{N}}

\begin{document}

\title[On a Free Boundary Problem for a Model of Polymeric Fluids]{Weak Dissipative solutions to a free-boundary problem for finitely extensible bead-spring chain molecules: variable viscosity coefficients}

\author[Donatelli]{D.\ Donatelli}
\address[Donatelli]{\newline
Departement of Information Engineering Computer Science and Mathematics\\
University of L'Aquila\\
67100 L'Aquila, Italy.}
\email[]{\href{donatella.donatelli@univaq.it}{donatella.donatelli@univaq.it}}
\urladdr{\href{http://people.disim.univaq.it/~donatell/}{http://people.disim.univaq.it/\~{}donatell}}

\author[Thorsen]{T.\ Thorsen}
\address[Thorsen]{\newline
Department of Mathematics \\ University of Maryland \\ College Park, MD 20742-4015, USA.}
\email[]{\href{http://www.math.umd.edu}{tthorsen@math.umd.edu}}

\author[Trivisa]{K. \ Trivisa}
\address[Trivisa]{\newline
Department of Mathematics \\ University of Maryland \\ College Park, MD 20742-4015, USA.}
\email[]{\href{http://www.math.umd.edu}{trivisa@math.umd.edu}}
\urladdr{\href{http://www.math.umd.edu/~trivisa}{math.umd.edu/\~{}trivisa}}

\date{\today}



\thanks{The research of D.D. and T.T. leading to these results was supported by the European Union's Horizon 2020 Research and Innovation Programme under the Marie Sklodowska-Curie Grant Agreement No 642768 (Project Name: ModCompShock). K.T. gratefully acknowledges the support by the National Science Foundation under the award DMS-1614964.}

\maketitle

\begin{abstract}
We investigate the global existence  of weak
solutions to a free boundary problem governing the evolution of  finitely extensible bead-spring chains in dilute polymers. The free boundary in the present context is defined with regard to  a density threshold of $\rho = 1,$ below which the fluid is modeled as compressible and above which the fluid is modeled as incompressible. 
The present article focuses on the  physically relevant case in which the viscosity coefficients present in the system depend on the  polymer number density, extending the earlier work \cite{DT2018}. 
We construct the weak solutions of
the free boundary problem  by performing the asymptotic limit as the adiabatic exponent $\gamma$
 goes to $\infty$ for the macroscopic model introduced by Feireisl, Lu and S\"{u}li in \cite{FLS} (see also \cite{BS2016}).
The weak sequential stability of the family of dissipative (finite energy) weak solutions to the free boundary problem is also established.

\end{abstract}

\tableofcontents{}

\section{Introduction}
Micro--macro models of dilute polymeric fluids  are typically derived using principles from  statistical physics and are based on the coupling of the Navier--Stokes system to the Fokker--Planck equation. 
This coupling demands naturally the development of new analytical techniques and  multiscale methods to analyze the flow of rheologically complex fluids. 
The multiscale models of such viscoelastic fluids bridge directly the microscopic scale of kinetic theory and the macroscopic scale of continuum mechanics.
In these models polymer molecules are idealized as chains of massless
beads, linearly connected with inextensible rods or elastic springs.

In the present work, we   investigate a free boundary problem for a polymeric fluid, defined by means of a pressure threshold above which the fluid is taken to be incompressible, and below which the fluid is compressible (cf.
Lions and Masmoudi \cite{LionsMasmoudi-1999}).   In \cite{DT2017}, \cite{DT2018} Donatelli and Trivisa established existence of weak solutions to such free boundary problems for two distinct models of polymeric fluids, the Doi model and the FENE model, both of which consider a dilute solution of polymers in a fluid solvent. In the Doi model the polymers are taken to be inflexible rods, while in the FENE model the polymers are modeled as flexible chains of beads connected by finitely extensible, nonlinear, elastic springs. The microscopic models governing the evolution of the polymers are coupled with the macroscopic model for the fluid solvent, in this case the Navier-Stokes equations.

 In a series of papers (\cite{BS2011}-\cite{BS2016}) Barrett and S\"{u}li proved existence of weak solutions to an initial boundary value problem for the FENE model in the case of  both incompressible and compressible solvents. Bae and Trivisa proved existence of weak solutions to the Doi model in both the compressible \cite{BT2012} and incompressible \cite{BT2013} cases. Recently, Feireisl, Lu, and S\"{u}li \cite{FLS} proved a weak sequential stability result for the compressible FENE model when the viscosity coefficients for the solvent are dependent on the polymer number density. 

Motivated by physical considerations,  the goal of this paper is to present well-posedness results for a free boundary problem 
derived from the  Navier-Stokes-Fokker-Plack system for polymeric fluids by taking the limit as the adiabatic exponent $\gamma$  approaches $\infty$ in the case of variable viscosity coefficients.
More precisely, the viscosity coefficients under consideration depend on the polymer number  density  as in \cite{FLS}.

The main ingredients of our approach can be formulated as follows:
\begin{itemize}
\item A suitable variational formulation of the underlying physical principles based on the dissipation of energy.
\item Physically grounded structural hypotheses imposed on the viscous stress tensor as well as the elastic extra-stress tensor in the system.
\item Extension of the multipliers technique of Lions, which now requires new delicate estimates in order to accommodate the variable viscosities and the loss of regularity of the sequence of  approximate velocities ${u_n}.$
\end{itemize}

The main contribution to the existing theory, and the principal new difficulties to
be dealt with can be characterized as follows:
\begin{itemize}
\item  We construct a sequence of approximating problems $(\vc{P_n})$. These approximating problems will be taken to be the compressible problem described in Section 2, with adiabatic exponents $\gamma_n$ such that $\gamma_n \to \infty$,
\item We utilize stability results for the compressible problem (see \cite{FLS}) to demonstrate convergence of the approximating solutions to the solution of the problem $(\vc{P_F})$.
\item In order to accommodate the variable viscosity coefficients and the loss of regularity for the approximate velocity sequence ${u_n}$ delicate commutator estimates need to be established.
\end{itemize}


The paper is structured as follows. Section \ref{sec:model}  presents the modeling assumptions and governing equations for the polymeric fluid, along with notation and definitions that will be used throughout the paper. Section \ref{sec:FBP} introduces  the free boundary problem and the notion of a weak solution as well as  the main existence result (Theorem \ref{thm:FBexist}). 
Section \ref{sec:approx} is dedicated to the construction of approximating problems, presents the notion of their solutions, and states the existence result for these solutions.
 Section \ref{sec:proof} presents the proof of the main result (Theorem \ref{thm:FBexist}), which  is a consequence of Theorem \ref{thm:stab}. The rest of the section is dedicated to proving Theorem \ref{thm:stab}, which involves obtaining: 
(a) a priori estimates via an energy inequality for the approximating problems; (b)  a uniform $L^1$ bound on the quantity $\rho_n^{\gamma_n}$; (c)  the convergence results stated in Section \ref{sec:conv}, which 
are established as a consequence of these two ingredients; (d)  the free boundary conditions satisfied by the limiting solution (Subsection \ref{sec:FBC}).
Finally, Section \ref{SWSS}  presents  Theorem \ref{WSS} which establishes the weak sequential stability of  the family of dissipative solutions to the free boundary problem Problem $(\vc{P_F}).$

\section{Modeling}\label{sec:model}
We first consider a model for a general polymeric fluid consisting of a compressible, isothermal, barotropic, viscous Newtonian fluid solvent in a solution with polymers modeled as flexible bead-spring chains. We make several assumptions:
\begin{itemize}
\item [(i)] The fluid occupies a bounded Lipschitz domain $\Omega \subset \R^3$.
\item [(ii)] The polymers are modeled as linear chains of $K + 1$ beads connected by $K$ finitely extensible, nonlinear, elastic (FENE) springs.
\item [(iii)] The polymer solution is dilute.
\item [(iv)] The drag coefficient $\zeta = 1$ is constant.
\item [(v)] There are no external body forces acting on the fluid.
\end{itemize}
Under these assumptions, the evolution of the fluid is modeled by the compressible Navier-Stokes equations
\begin{equation*}
\begin{gathered}
\partial_t \rho + \Div_x(\rho \vu) = 0,\\
\partial_t(\rho \vu) + \Div_x(\rho \vu \otimes \vu) + \nabla_x p(\rho) - \Div_x \mathbb{S} = \Div_x \mathbb{T},
\end{gathered}
\end{equation*}
where $\vu$ is the fluid velocity, $\rho$ is the fluid density, $p$ is the fluid pressure, $\mathbb{S}$ is the viscous stress tensor, and $\mathbb{T}$ is an elastic extra-stress tensor. The addition of the term involving the elastic extra-stress tensor $\mathbb{T}$ is due to the fluid-polymer interactions. 
We assume the following pressure law
\[
p(\rho) = \rho^{\gamma},
\]
while the viscous stress tensor $\mathbb{S}$ is the Newtonian stress tensor defined by
\begin{equation*}
\mathbb{S}[\eta, \vu] = \mus\left[\frac{\nabla_x \vu + \nabla_x^T \vu}{2} - \frac{1}{3}( \Div_x \vu) \mathbb{I}\right] + \mub(\Div_x \vu) \mathbb{I}.
\end{equation*}
Here, $\mu^S$ and $\mu^B$ are the shear and bulk viscosity coefficients, respectively. In previous studies (\cite{DT2017}), these coefficients have been taken to be constants. The present article treats the physically relevant case in which the viscosity coefficients are functions of the polymer number density which brings addtional stumbling blocks in the analysis of the problems of existence and weak sequential stability.

Each spring in the chain can be modeled by a conformation vector $\vq_i$, which represents the orientation and extension of the spring. Since the springs are finitely extensible, each spring has a maximal extension length $r_i^{1/2}$, so each conformation vector $\vq_i$ belongs to the domain $D_i = B(0, r_i^{1/2}) \subset \R^3$. Then, the entire chain can be modeled by the conformation vector $\vq = (\vq_1^T, ..., \vq_K^T)^T$, which belongs to the domain $D = D_1 \times \hdots \times D_K$. Additionally, there is a spring potential $U_i \in C^1([0, \frac{r_i}{2}))$  associated with each spring such that $U_i(0) = 0$, $\lim_{s \to \frac{b_i}{2}}U_i(s) = \infty$. The $i^{th}$ partial Maxwellian $M_i(\vq_i)$ is defined by
\begin{equation*}
M_i(\vq_i) = \frac{1}{\mathcal{Z}_i} e^{-U_i'\left(\frac{|\vq_i|^2}{2}\right)}, \ \ \mathcal{Z}_i = \int_{D_i}e^{-U_i'\left(\frac{|\vq_i|^2}{2}\right)},
\end{equation*}
while the total Maxwellian $M(\vq)$ is given by
\begin{equation*}
M(\vq) = \prod_{i = 1}^K M_i(\vq_i).
\end{equation*}
The polymer probability density function $f = f(t, x, \vq)$ is defined such that $f(t, x, \vq) \dq$ denotes the probability that a polymer with center of mass $x$ at time $t$ has a conformation vector $\vq$ in the domain $\dq$. The evolution of the polymer probability density function $f$ is then governed by the Fokker-Planck equation 
\begin{align}\label{eq:1}
\begin{split}
&\partial_t f + \Div_x(f\vu) + \sum_{i = 1}^K \Div_{\vq_i}((\nabla_x \vu)\vq_i \ f) \\
& \quad \quad= \epsilon \Delta_x f + \frac{1}{4\lambda}\sum_{i = 1}^K \sum_{j = 1}^K A_{ij} \Div_{\vq_i}\left(M \nabla_{\vq_j}\left(\frac{f}{M}\right)\right),
\end{split}
\end{align}
where, $\epsilon$ is a center-of-mass diffusion coefficient, $\lambda$ is the Deborah number which characterizes the elastic relaxation of the fluid, and $A = (A_{ij})_{i,j = 1}^K$ is the positive-definite Rouse matrix, which describes the connectivity of the bead-spring chain. For a more thorough derivation of \eqref{eq:1}, see \cite{BS2012B}. We define the polymer number density $\eta(t, x)$ by
\begin{equation*}
\eta(t, x) = \int_D f(t, x, \vq) \ {\rm d} \vq,
\end{equation*}
which obeys an advection-diffusion equation
\begin{equation*}
\partial_t \eta + \Div_x(\eta \vu) = \epsilon\Delta_x \eta,
\end{equation*}
obtained by a formal integration of \eqref{eq:1} over the domain $D$ under the zero-penetration boundary conditions \eqref{eq:FBPBC2}. Additionally, we assume that the elastic extra-stress tensor has the form
\begin{equation*}
\mathbb{T} = \mathbb{T}_1 + \mathbb{T}_2,
\end{equation*}
where
\begin{equation*}
\mathbb{T}_1 = k \left[\sum_{i = 1}^K \mathbb{C}_i(f) - (K + 1) \left(\int_D f \ d\vq\right)\mathbb{I}\right]
\end{equation*}
is the standard Kramer's expression and
\begin{equation*}
\mathbb{T}_2 =\left( \int_{D\times D}\gamma(\vq, \vq')f(t, x, \vq)f(t, x, \vq') \ d\vq \ d\vq'\right)\mathbb{I}.
\end{equation*}
Here, $\gamma$ is an interaction kernel and
\begin{equation*}
\mathbb{C}_i(f) := \int_D fU_i'\left(\frac{|\vq_i|^2}{2}\right)\vq_i\vq_i^T \ d\vq.
\end{equation*}
Under the assumption that $\gamma = \delta > 0$ is a constant, the extra-stress tensor reduces to the form
\begin{equation*}
\mathbb{T} =  k\sum_{i = 1}^K \mathbb{C}_i(f) - \left(k(K + 1) \eta + \delta \eta^2\right)\mathbb{I}.
\end{equation*}

\subsection{Notation and Definitions}
Here we summarize notation that will be used throughout the paper.
\begin{itemize}
\item $\tilde{f} = f/M$.
\item $A_0$ is the smallest eigenvalue of the Rouse matrix $A$.
\item $\mathcal{F}(s) = s(\log s - 1) + 1$.
\item $\mathcal{M}((0, T)\times \Omega)$ is the space of bounded measures on $(0, T)\times \Omega$.
\item $\mathbb{I}$ denotes the $3\times 3$ identity tensor.
\item We write $x \lesssim y$ when there exists a constant $C$ such that $x \leq Cy$, and we write $x \lesssim_T y$ when there exists a constant $C(T)$, dependent only on time, such that $x \leq C(t) y$.
\item We use $\to$ to denote strong convergence and $\rightharpoonup$ to denote weak convergence.
\item $L_M^r(D)$ is the Maxwellian-weighted Lebesgue space defined by the norm
\[
\|f\|_{L_M^r(D)} = \left(\int_0^T M|f|^r \ d\vq\right)^{1/r}.
\]
Similarly, we define $L_M^r(\Omega \times D) = L^r(\Omega; L_M^r(D))$.
\item $Z_r = \{f \in L_M^r(\Omega \times D); f \geq 0 \text{ a.e. on }\Omega \times D\}$.
\item $M^{-1}(H^s(\Omega \times D))'= \{M^{-1}f: f \in (H^s(\Omega \times D))'\}$, where $M$ is the Maxwellian.
\item $C_w(0, T; X)$ is the space of weakly continuous functions over $X$, i.e. the space of functions  $v \in L^\infty(0, T; X)$ such that, for all $w \in X'$, the mapping $t \mapsto\langle w, v(t)\rangle_X$ is continuous.
\end{itemize}
Additionally, we define several operators which will be used in the analysis of the problem.
\begin{definition} For $p \in (1, \infty)$ and $\Omega \subset \R^3$ a bounded Lipschitz domain, the bounded linear operator $\mathcal{B}: L_0^p(\Omega) \to W_0^{1,p}(\Omega)$ is defined such that
\[
\Div_x \mathcal{B}(f) = f, \ \ \|\mathcal{B}(f)\|_{W_0^{1,p}(\Omega)} \leq c(p, \Omega) \|f\|_{L^p(\Omega)} \text{ for all }f \in L_0^p(\Omega),
\]
where $L_0^p(\Omega)$ is the space of all $L^p(\Omega)$ functions with zero mean. Additionally, if $f = \Div_x g$ for $g \in L^q(\Omega)$, where $q \in (1, \infty)$ and $g \cdot n = 0$ on $\partial \Omega$, it follows that
\[
\|\mathcal{B}(f)\|_{L^q(\Omega; \R^3)} \leq c(p, \Omega) \|g\|_{L^q(\Omega; \R^3)}.
\]
\end{definition}

\begin{definition}[Riesz Operator] We define the following operators on $\R^3$:
\[
\mathcal{R}_{ij} = \partial_i \partial_j \Delta^{-1}, \ \ \mathcal{A}_j = -\partial_j \Delta^{-1}.
\]
Then we have the following properties:
\begin{itemize}
\item [(i)] $\mathcal{R}_{ij} = -\partial_i \mathcal{A}_j$,
\item [(ii)] $\sum_j\mathcal{R}_{jj} = -\sum_j \partial_j \mathcal{A}_j = \mathbb{I}$,
\item [(iii)] For any $p \in (1, \infty)$, $R_{ij}$ is a bounded operator from $L^p(\R^3)$ to $L^p(\R^3)$.
\item [(iv)] For any $p \in (1, \infty)$, $u \in L^p(\R^3), v \in L^q(\R^3)$, such that $\frac{1}{p} + \frac{1}{q} = 1$, we have
\[
\int_{\R^3} \mathcal{R}_{ij}[u]v \dx = \int_{\R^3} u \mathcal{R}_{ij}[v].
\]
\item [(v)] For any $p \in (1, 3)$, $\mathcal{A}_j$ is a bounded operator from $L^p(\R^3)$ to $L^{\frac{3p}{3 - p}}(\R^3)$.
\end{itemize}
\end{definition}
Here, we also note that the partial derivatives $\partial_i, \partial_j$ commute with the inverse Laplacian operator $\Delta^{-1}$, so we can write
\[
\mathcal{R}_{ij} = \Delta^{-1} \partial_i \partial_j, \ \ \mathcal{A}_j = -\Delta^{-1} \partial_j.
\]

\section{Free Boundary Problem} \label{sec:FBP}
We now define a free boundary problem for the polymeric fluid model described in the previous section. We implement a density threshold of $\rho = 1$, below which the fluid is modeled as compressible and above which the fluid is modeled as incompressible. In the compressible regime we assume that the fluid pressure $\pi$ vanishes.
\subsection{Governing Equations}
The free boundary problem ($\vc{P_F}$) is defined by the system of governing equations
\begin{equation}\label{eq:FBcont}
\partial_t \rho + \Div_x(\rho \vu) = 0
\end{equation}
\begin{equation}\label{eq:FBmom}
\partial_t(\rho \vu) + \Div_x(\rho \vu \otimes \vu) + \nabla_x \pi - \Div_x \mathbb{S} = \Div_x \mathbb{T} ,
\end{equation}
\begin{align}\label{eq:FBFP}
\begin{split}
&\partial_t f + \Div_x(f\vu) + \sum_{i = 1}^K \Div_{\vq_i}((\nabla_x \vu)\vq_i \ f)\\
& \quad \quad = \epsilon \Delta_x f + \frac{1}{4\lambda}\sum_{i = 1}^K \sum_{j = 1}^K A_{ij} \Div_{\vq_i}\left(M \nabla_{\vq_j}\left(\frac{f}{M}\right)\right),
\end{split}
\end{align}
\begin{equation}\label{eq:FBdiff}
\partial_t \eta + \Div_x (\eta \vu) = \epsilon \Delta_x \eta,
\end{equation}
supplemented with the free boundary conditions
\begin{equation}\label{eq:FBC1}
\Div_x \vu = 0 \text{ a.e. on }\{\rho = 1\},
\end{equation}
\begin{equation}\label{eq:FBC2}
\pi \geq 0 \text{ a.e. on }\{\rho = 1\},
\end{equation}
\begin{equation}\label{eq:FBC3}
\pi = 0 \text{ a.e. on }\{\rho < 1\}.
\end{equation}
The nonconstant viscosity coefficients $\mu^S, \mu^B$ are now functions of the polymer number density $\eta$ and are defined by
\[
\mu^S = \phi_S(\eta),\ \ \mu^B = \phi_B(\eta)
\]
for some functions $\phi_S, \phi_B \in C^1([0, \infty))$. We fix the following growth conditions:
\begin{equation}\label{eq:omegacond}
\begin{gathered}
c_1(1 + s)^\omega \leq \phi_S(s) \leq c_2(1 + s)^\omega, \ \ |\phi_S'(s)| \leq c_3 + c_4(1 + s)^{\omega - 1}, \\
\ \ 0 \leq \phi_B(s) \leq c_5(1 + s)^\omega,
\end{gathered}
\end{equation}
where $\omega \in \R$ and $c_1, ..., c_5$ are positive real constants. These growth constraints allow us to bound the viscosity coefficients $\mu^S, \mu^B$ along with their inverses. Such bounds are necessary to obtain regularity results for the fluid velocity.
\subsubsection{Boundary Conditions}
We define $\partial \overline{D}_i = D_1 \times \hdots \times D_{i - 1} \times  \partial D_i \times D_{i + 1} \times\hdots \times D_K$ and note that $\vq_i$ is normal to $\overline{D}_i$. We impose the boundary conditions
\begin{equation} \label{eq:FBPBC1}
\vu = 0 \text{ on }\partial \Omega,
\end{equation}
\begin{align}\label{eq:FBPBC2}
\begin{split}
&\left[\frac{1}{4\lambda} \sum_{j = 1}^K A_{ij}\nabla_{\vq_j}\left(\frac{f}{M}\right) -( \nabla_x \vu)\vq_i \ f\right]\cdot \frac{\vq_i}{|\vq_i|} = 0, \\
& \quad \quad \quad \text{ on } \Omega \times \partial \overline{D}_i \times (0, T], \ \ i = 1, ..., K,
\end{split}
\end{align}
and
\begin{equation}\label{eq:FBPBC3}
\nabla_x f \cdot \bm{n} = 0,\ \  \nabla_x\eta \cdot \bm{n} = 0 \text{ on } \partial \Omega \times D \times (0, T],
\end{equation}
where $\bm{n}$ is normal to $\partial \Omega$.

\subsubsection{Initial Data}
The system is also supplemented with initial data 
\[
\rho_0 = \rho|_{t = 0},\ \  m_0 = \rho\vu|_{t = 0}, \ \ f_0 = f|_{t = 0}, \ \ \eta_0 = \eta|_{t = 0}. 
\]
We require that the initial data satisfy the following conditions
\begin{equation} \label{eq:FBIC}
\begin{gathered}
0 \leq \rho_0 \leq 1 \text{ a.e. in }\Omega, \ \ \rho_0 \in L^1(\Omega), \ \ \rho_0 \not\equiv 0,\\
\fint_\Omega \rho_0 = V < 1,\\
m_0 \in L^{2}(\Omega),\ \  m_0 = 0 \text{ a.e. on } \{\rho_0 = 0\},\\
\rho_0|\vu_0|^2 \in L^1(\Omega),\\
\vu_0 = \frac{m_0}{\rho_0} \text{ on }\{\rho_0 > 0\}, \ \ \vu_0 = 0 \text{ on } \{\rho_0 = 0\},\\ 
f_0 \geq 0 \text{ a.e. in } \Omega \times D,\ \  f_0\log \frac{f_0}{M} \in L^1(\Omega \times D),\\
\eta_0 = \int_D f_0 \text{ a.e. in }\Omega, \ \ \eta_0 \in L^2(\Omega).
\end{gathered}
\end{equation}

\subsection{Weak Dissipative Solutions and Main Result}
We are now able to rigorously define a weak solution to the free boundary problem ($\vc{P_F}$).
\begin{definition}[Weak  Dissipative Solution to ($\vc{P_F}$)] \label{def:FBsol} A vector $(\rho, \vu, \pi, f)$ is a weak dissipative solution to the problem ($\vc{P_F}$) provided
\begin{itemize}
\item [(i)] The following regularity results hold:
\begin{gather*}
\rho \in C([0, T]; L^p(\Omega)), \ 1 \leq p < \infty,\\
\vu \in L^r(0, T; W_0^{1,r}(\Omega)) \text{ for some }r > 1, \ \ \rho |\vu|^2 \in L^\infty(0, T; L^1(\Omega)),\\
\pi \in \mathcal{M}((0, T)\times \Omega),\\
\eta \in L^\infty(0, T; L^2(\Omega)) \cap L^2(0, T; {H}^1(\Omega)),\\
\tilde{f} \in L^p(0, T; Z_1) \cap H^1(0, T; M^{-1}(H^s(\Omega \times D))'), \ \ 1\leq p < \infty, \ \ s > 1+\frac{1}{2}(K + 1)d
\end{gather*}
and $\pi$ is sufficiently regular that the condition $\pi(\rho - 1) = 0$ is satisfied in the sense of distributions.
\item [(ii)] The equations \eqref{eq:FBcont} - \eqref{eq:FBdiff} are satisfied in the sense of distributions.
\item [(iii)] The divergence-free condition $\Div_x \vu = 0$ is satisfied a.e. on $\{\rho = 1\}$.
\item [(iv)] The constraint $0 \leq \rho \leq 1$ is satisfied a.e. in $(0, T)\times \Omega$.
\item [(v)] In addition the weak solutions are dissipative in the sense that they satisfy the energy inequality
\begin{align} \label{eq:FBenergy}
\begin{split}
&\int_\Omega \left[\frac{1}{2}\rho|\vu|^2 + \delta \eta^2 + k \int_D M \mathcal{F}(\tilde{f}) \ {\rm d}\vq\right](t, \cdot ) \dx\\
& \quad + \int_0^t \int_\Omega \mus\left|\frac{\nabla_x \vu + \nabla_x^T\vu}{2} - \frac{1}{3}(\Div_x \vu)\mathbb{I}\right|^2 + \mub|\Div_x\vu|^2 \dxdt' \\
& \quad + 2\epsilon\delta \int_0^t \int_\Omega |\nabla_x \eta|^2 \dxdt' + \epsilon k \int_0^t\int_\Omega \int_D M |\nabla_x \sqrt{\tilde{f}}|^2 {\rm d} \vq \ \dxdt'\\
& \quad + \frac{kA_0}{4\lambda}\int_0^t\int_\Omega \int_D M \left|\nabla_{\vq}\sqrt{\tilde{f}}\right|^2 \ {\rm d}\vq \ \dxdt'\\
& \leq \int_\Omega \left[\frac{1}{2}\rho_{0}|\vu_{0}|^2 + \delta \eta_{0}^2 + k\int_D M \mathcal{F}(\tilde{f}_{0}) \ {\rm d}\vq\right]\dx.
\end{split}
\end{align}
\end{itemize}
\end{definition}
The main goal of this paper is to prove the following existence result for weak solutions to the free boundary problem ($\vc{P_F}$).
\begin{theorem}[Existence of Solutions to ($\vc{P_F}$)] \label{thm:FBexist} Fix $-\frac{4}{3} < \omega <\frac{5}{3}$. Suppose that the initial conditions \eqref{eq:FBIC} and the boundary conditions \eqref{eq:FBPBC1}-\eqref{eq:FBPBC3} are satisfied. Then, there exists a weak solution, in the sense of Definition \ref{def:FBsol}, to the problem $(\vc{P_F})$. Moreover, the energy inequality \eqref{eq:FBenergy} is satisfied.
\end{theorem}
The outline of the proof of Theorem \ref{thm:FBexist} is as follows:
\begin{itemize}
\item we construct a sequence of approximating problems ($\vc{P_n}$). These approximating problems will be taken to be the compressible problem described in Section \ref{sec:model}, with adiabatic exponents $\gamma_n$ such that $\gamma_n \to \infty$, and
\item we build upon the analysis in \cite{FLS} in order to demonstrate convergence of the approximating solutions to the solution of the problem ($\vc{P_F}$). 
\end{itemize}
In the next section, we set up the approximating problems ($\vc{P_n}$). We will define weak solutions to the problems ($\vc{P_n}$) and discuss the existence of weak solutions.

\section{Approximating Problems} \label{sec:approx}
The approximating problems ($\vc{P_n}$) will be defined by the governing equations
\begin{equation}\label{eq:cont}
\partial_t \rho_n + \Div_x(\rho_n \vu_n) = 0
\end{equation}
\begin{equation}\label{eq:mom}
\partial_t(\rho_n \vu_n) + \Div_x(\rho_n \vu_n \otimes \vu_n) + \nabla_x p_n(\rho_n) - \Div_x \mathbb{S}_n = \Div_x \mathbb{T}_n ,
\end{equation}
\begin{align}\label{eq:FP}
\begin{split}
&\partial_t f_n + \Div_x(f_n\vu_n) + \sum_{i = 1}^K \Div_{\vq_i}((\nabla_x \vu_n)\vq_i \ f_n) \\
& \quad \quad= \epsilon \Delta_x f_n + \frac{1}{4\lambda}\sum_{i = 1}^K \sum_{j = 1}^K A_{ij} \Div_{\vq_i}\left(M \nabla_{\vq_j}\left(\frac{f_n}{M}\right)\right),
\end{split}
\end{align}
\begin{equation}\label{eq:diff}
\partial_t \eta_n + \Div_x (\eta_n \vu_n) = \epsilon \Delta_x \eta_n,
\end{equation}
where
\begin{equation}
p_n(\rho_n) = \rho_n^{\gamma_n},
\end{equation}
and the adiabatic exponent $\gamma_n \to \infty$ as $n \to \infty$. We define the viscous and elastic stress tensors by
\begin{equation}
\mathbb{S}_n :=\muns \left(\frac{\nabla_x \vu_n + \nabla_x^T \vu_n}{2} - \frac{1}{3}(\Div_x \vu_n) \mathbb{I}\right) + \munb (\Div_x \vu_n) \mathbb{I}
\end{equation}
\begin{equation} \label{eq:extrastress}
\mathbb{T}_n:= k \sum_{i = 1}^K \mathbb{C}_i(f_n) - (k(K + 1) \eta_n + \delta \eta_n^2) \mathbb{I}.
\end{equation}
Here,
\[
\mathbb{C}_i(f_n):= \int_D f_n U_i'\left(\frac{|\vq_i|^2}{2}\right) \vq_i \vq_i^T \dq,
\]
as in Section \ref{sec:model}, and 
\[
\mu^S_n = \phi_S(\eta_n), \ \ \mu^B_n = \phi_B(\eta_n),
\]
where $\phi_S, \phi_B$ are the same $C^1$ functions as in the previous section. 

\subsubsection{Boundary Conditions}
We impose the same boundary conditions as in the free boundary problem, specifically
\begin{equation} \label{eq:BC1}
\vu_n = 0 \text{ on }\partial \Omega,
\end{equation}
\begin{equation}\label{eq:BC2}
\begin{gathered}
\left[\frac{1}{4\lambda} \sum_{j = 1}^K A_{ij}\nabla_{\vq_j}\left(\frac{f_n}{M}\right) -( \nabla_x \vu_n)\vq_i \ f_n\right]\cdot \frac{\vq_i}{|\vq_i|} = 0,\\ \text{ on } \Omega \times \partial \overline{D}_i \times (0, T], \ \ i = 1, ..., K,
\end{gathered}
\end{equation}
and
\begin{equation}\label{eq:BC3}
\nabla_x f_n \cdot \bm{n} = 0,\ \  \nabla_x\eta_n \cdot \bm{n} = 0 \text{ on } \partial \Omega \times D \times (0, T].
\end{equation}

\subsubsection{Initial Data}
The approximating problem is also supplemented with initial data 
\[
\rho_{n,0} = \rho_n|_{t = 0}, \ \ m_{n,0} = (\rho_n \vu_n)|_{t = 0}, \ \ f_{n,0} = f_n|_{t = 0},\ \  \eta_{n,0} = \eta_n|_{t = 0},
\]
which we assume satisfy the conditions
\begin{equation} \label{eq:IC}
\begin{gathered}
\rho_{n,0} \geq 0 \text{ a.e. in }\Omega, \ \ \rho_{n,0} \in L^{\gamma_n}(\Omega),\\
m_{n,0} \in L^{\frac{2\gamma_n}{\gamma_n + 1}}(\Omega),\ \  \rho_{n,0}|\vu_{n,0}|^2 \in L^1(\Omega),\\
f_{n,0} \geq 0 \text{ a.e. in } \Omega \times D,\ \  f_{n,0}\log \frac{f_{n,0}}{M} \in L^1(\Omega \times D),\\
\eta_{n,0} = \int_D f_{n,0} \text{ a.e. in }\Omega, \ \ \eta_{n,0} \in L^2(\Omega).
\end{gathered}
\end{equation}
These initial conditions are sufficient to demonstrate existence of weak solutions to the problem ($\vc{P_n}$). We also impose further conditions on the initial data in order to guarantee convergence of the sequence of approximating solutions:
\begin{equation} \label{eq:IC2}
\begin{gathered}
\rho_{n,0} \rightharpoonup \rho_0 \text{ in } L^1(\Omega),\\
m_{n,0} \rightharpoonup m_0 \text{ in }L^2(\Omega),\\
f_{n,0} \rightharpoonup f_0 \text{ in }L^1(\Omega\times D),\\
\eta_{n,0} \rightharpoonup \eta_0 \text{ in }L^2(\Omega),\\
\rho_0 |\vu_{n,0}|^2 \text{ uniformly bounded in }L^1(\Omega),\\
f_{0,n} \log \frac{f_{0,n}}{M} \text{ uniformly bounded in }L^1(\Omega\times D),
\end{gathered}
\end{equation}
\begin{equation} \label{eq:IC3}
\int_{\Omega} \rho_{n,0}^{\gamma_n} \leq c \gamma_n
\end{equation}
for some fixed $c > 0$, independent of $n$, and
\begin{equation}\label{eq:IC4}
\frac{1}{|\Omega|}\int_{\Omega} \rho_{n,0} = V_n,
\end{equation}
with $0 < V_n \le V< 1$ and $\lim_{n \to \infty} V_n = V$. These  constraints on the initial data are necessary to ensure that the limit solution corresponds to the appropriate initial data, and to uniformly bound the initial energy of the system independently from $n$.

\subsection{Weak Solutions}
For fixed $n$ we define the following notion of a weak solution to the problem ($\vc{P_n}$).
\begin{definition}[Weak Solution to ($P_n$)] \label{def:sol} A vector $(\rho_n, \vu_n, f_n, \eta_n)$ is a solution to the problem ($\vc{P_n}$) provided
\begin{itemize}
\item [(i)] The following regularity results hold:
\begin{equation*}
\begin{gathered}
\rho_n \geq 0 \text{ a.e., }\rho_n \in C_w([0, T]; L^{\gamma_n}(\Omega)),\\
\vu_n \in L^r(0, T; W_0^{1,r}(\Omega)) \text{ for some }r > 1, \ \ \rho_n\vu_n \in C_w([0, T]; L^{\frac{2\gamma_n}{\gamma_n + 1}}(\Omega)),\\
\rho_n |\vu_n|^2 \in L^\infty(0, T; L^1(\Omega)),\\
f_n \geq 0 \text{ a.e., }f_n \in C_w([0, T]; L^1(\Omega \times D)),\\
\nabla_x f_n \in L^1((0, T)\times \Omega \times D), \ \ M \nabla_{\vq}\tilde{f}_n \in L^1((0, T)\times \Omega \times D),\\
\eta_n = \int_D f_n \ {\rm d}\vq \text{ a.e., } \eta_n \in C_w([0, T];L^2(\Omega)) \cap L^2(0, T; W^{1,2}(\Omega)),\\
\mathbb{T}_n \text{ satisfies \eqref{eq:extrastress} a.e.}, \ \ \mathbb{T}_n \in L^1((0, T) \times \Omega).
\end{gathered}
\end{equation*}
\item [(ii)] The equations \eqref{eq:cont}-\eqref{eq:diff} are satisfied in the sense of distributions.
\item [(iii)] The continuity equation \eqref{eq:cont} is satisfied in the sense of renormalized solutions, i.e. for all $b \in C^1([0, \infty))$ such that $|b(s)| + |sb'(s)| \leq c < \infty$ for all $s \in [0, \infty)$, the equality
\begin{equation} \label{eq:renorm}
\partial_t b(\rho_n) + \Div_x(b(\rho_n) \vu_n) + (b(\rho_n) - \rho_nb'(\rho_n)) \Div_x \vu_n = 0
\end{equation}
holds in the sense of distributions.
\item [(iv)] The following energy inequality holds:
\begin{align} \label{eq:energy}
\begin{split}
&\int_\Omega \left[\frac{1}{2}\rho_n|\vu_n|^2 + \frac{\rho_n^{\gamma_n}}{\gamma_n - 1} + \delta \eta_n^2 + k \int_D M \mathcal{F}(\tilde{f}_n) \ {\rm d}\vq\right](t, \cdot ) \dx\\
& \quad + \int_0^t \int_\Omega \muns\left|\frac{\nabla_x \vu_n + \nabla_x^T\vu_n}{2} - \frac{1}{3}(\Div_x \vu_n)\mathbb{I}\right|^2 + \munb|\Div_x\vu_n|^2 \dxdt' \\
& \quad + 2\epsilon\delta \int_0^t \int_\Omega |\nabla_x \eta_n|^2 \dxdt' + \epsilon k \int_0^t\int_\Omega \int_D M |\nabla_x \sqrt{\tilde{f}_n}|^2 {\rm d} \vq \ \dxdt'\\
& \quad + \frac{kA_0}{4\lambda}\int_0^t\int_\Omega \int_D M \left|\nabla_{\vq}\sqrt{\tilde{f}_n}\right|^2 \ {\rm d}\vq \ \dxdt'\\
& \leq \int_\Omega \left[\frac{1}{2}\rho_{n,0}|\vu_{n,0}|^2 + \frac{\rho_{n,0}^{\gamma_n}}{\gamma_n - 1} + \delta \eta_{n,0}^2 + k\int_D M \mathcal{F}(\tilde{f}_{n,0}) \ {\rm d}\vq\right]\dx.
\end{split}
\end{align}
\end{itemize}
\end{definition}
For fixed $n$ the existence of such weak solutions is inferred by the analysis in \cite{FLS} under certain conditions on $\omega, \gamma_n$. The result reads as follows.
\begin{theorem}[Existence of Solutions to ($\vc{P_n}$)] \label{thm:exist}
Assume that the initial data $(\rho_{n,0}, m_{n,0}, f_{n,0}, \eta_{n,0})$ satisfy the conditions \eqref{eq:IC}, and that the boundary conditions \eqref{eq:BC1} - \eqref{eq:BC2} are satisfied. If, in addition, either
\[
\gamma_n > \frac{3}{2} \text{ and } 0 \leq \omega < \frac{5}{3} \text{, or } \gamma_n > \frac{6}{4 + 3\omega} \text{ and }-\frac{4}{3} < \omega \leq 0,
\]
then there exists a weak solution $(\rho_n, \vu_n, f_n, \eta_n)$ to the problem ($\vc{P_n}$), in the sense of Definition \ref{def:sol}, corresponding to the initial data $(\rho_{n,0}, m_{n,0}, f_{n,0}, \eta_{n,0})$.
\end{theorem} 
It is important to point out that in \cite{FLS} only a stability result is proved. The existence can be established by combining the stability result with 
the existence proof established by Barrett and S\"{u}li (\cite{BS2016}) in the case of constant viscosity coefficients. We refer the reader to \cite{FLS}, and outline only the main steps here.
\begin{itemize}
\item Construct a sequence of approximating problems. This is done through the introduction of a cutoff function applied to the probability density function $f$, as well as through regularizing terms added to the fluid pressure and the continuity equation. The problem is then discretized in time.
\item Prove existence of weak solutions to the approximating problems through fixed point-type arguments.
\item Utilize the stability results from \cite{FLS} and the methods from \cite{DT2017} to demonstrate convergence of the approximating solutions to a weak solution to the compressible problem.
\end{itemize}

\section{Proof of Main Theorem} \label{sec:proof}
This section is devoted to the proof of the main Theorem \ref{thm:FBexist}. We first state the following stability result.

\begin{theorem}[Convergence of Approximating Solutions] \label{thm:stab} Fix $-\frac{4}{3} < \omega < \frac{5}{3}$, and let $\{\gamma_n\}_{n = 1}^\infty$ be a sequence of real numbers such that $\gamma_n \to \infty$ as $n \to \infty$, and for all $n \in \mathbb{N}$,
\begin{equation}\label{eq:gammaomega}
\gamma_n > \frac{3}{2} \text{ if }\omega \geq 0, \ \ \gamma_n > \frac{6}{4 + 3\omega} \text{ if }\omega \leq 0.
\end{equation}
Let $\{(\rho_{n,0}, \vu_{n,0}, f_{n,0}, \eta_{n,0})\}_{n = 1}^\infty$ be a sequence of initial data satisfying the initial conditions \eqref{eq:IC}  and \eqref{eq:IC2}. Then, for each $n$ there exist a global weak solution $(\rho_n, \vu_n, f_n, \eta_n)$ to the problem $(\vc{P_n})$ (in the sense of Definition \ref{def:sol}), corresponding to initial data $(\rho_{n,0}, \vu_{n,0}, f_{n,0}, \eta_{n,0})$, such that
\begin{equation*}
\lim_{n \to \infty}(\rho_n - 1)_+ = 0 \text{ in }L^\infty(0, T; L^p(\Omega)) \text{ for any }1 \leq p < \infty.
\end{equation*}
Moreover, 
\begin{equation*}
(\rho_n)^{\gamma_n} \text{ is bounded in }L^1 \text{ for }n \text{ such that } \gamma_n \geq 4,
\end{equation*}
and up to a subsequence there exists $\pi \in \mathcal{M}((0, T)\times \Omega)$ such that
\begin{equation*}
(\rho_n)^{\gamma_n} \rightharpoonup \pi \text{ as }n \to \infty.
\end{equation*}
If, in addition, we assume that $\rho_{n,0} \to \rho_0$ in $L^1(\Omega)$, then we have the following convergence (up to a subsequence):
\begin{equation} \label{eq:conv}
\begin{gathered}
\rho_n \to \rho \text{ in }C_w([0, T];L^p(\Omega)) \text{ for any }1 \leq p < \infty,\\
\rho_n \vu_n \to \rho \vu \text{ in }C_w([0, T]; L^r( \Omega)) \text{ for any } 1 \leq r < 2,\\
\rho_n \vu_n \otimes \vu_n \rightharpoonup \rho \vu \otimes \vu \text{ in } L^2(0, T; L^r(\Omega)) \text{ for some } r> 1,\\
\vu_n \rightharpoonup \vu \text{ in } L^2(0, T; W^{1,2}(\Omega)) \text{ if } \omega \geq 0,\\
\vu_n \rightharpoonup \vu \text{ in } L^{\frac{20}{10 + 3|\omega|}}(0, T; W^{1,{\frac{20}{10 + 3|\omega|}}}(\Omega)) \text{ if } \omega \leq 0,\\
f_n \to f \text{ in }L^1(0, T; L^1(\Omega \times D)),\\
\nabla_{\vq} \sqrt{\tilde{f}_n} \rightharpoonup \nabla_{\vq} \sqrt{\tilde{f}}, \ \ \nabla_x \sqrt{\tilde{f}_n} \rightharpoonup \nabla_x \sqrt{\tilde{f}} \text{ in }L^2(0, T; L_M^2(\Omega \times D)),\\
\eta_n \to \eta \text{ in }C_w(0, T; L^2(\Omega)) \text{ and weakly in }L^2(0, T; W^{1,2}(\Omega)),\\
(\mu_n^S, \mu_n^B) \to (\mu^S, \mu^B) \text{ in }L^q((0, T)\times \Omega) \text{ for any }1 \leq q < \infty \text{ when } \omega \leq 0,\\
(\mu_n^S, \mu_n^B) \to (\mu^S, \mu^B) \text{ in }L^\frac{10}{3\omega}((0, T)\times \Omega) \text{ when } \omega \geq 0,
\end{gathered}
\end{equation}
and $(\rho, \vu, \pi, f, \eta)$ is a weak solution to the problem $(\vc{P_F})$ in the sense of Definition \ref{def:FBsol}.
\end{theorem}
As we will see in the next section, the proof of Theorem \ref{thm:FBexist} is a consequence of Theorem \ref{thm:stab}. In Section \ref{sec:stabproof}, we will prove Theorem \ref{thm:stab}.

\subsection{Proof of Theorem \ref{thm:FBexist}}\label{sec:existproof}
For any initial data $(\rho_0, m_0, f_0, \eta_0)$ satisfying \eqref{eq:FBIC}, we can construct a sequence of initial data $\{(\rho_{n,0}, m_{n,0}, f_{n,0}, \eta_{n,0})\}_{n \in \N}$, and an accompanying sequence of adiabatic constants $\{\gamma_n\}_{n \in \N}$, satisfying the hypotheses of Theorem \ref{thm:stab}, and thus obtain the convergence results in \eqref{eq:conv}. In particular, for $n \geq 2$ we can take
\[
\gamma_n = n, \ \ \rho_{n,0} = \rho_0, \ \ m_{n,0} = m_0, \ \ f_{n,0} = f_0, \ \ \eta_{n,0} = \eta_0.
\]
In order to prove Theorem \ref{thm:FBexist}, it remains to show that the energy inequality \eqref{eq:FBenergy} is satisfied. We set
\[
E_{n,0} := \int_\Omega \left[\frac{1}{2}\rho_{n,0}|\vu_{n,0}|^2 + \frac{\rho_{n,0}^{\gamma_n}}{\gamma_n - 1} + \delta \eta_{n,0}^2 + k\int_D M \mathcal{F}(\tilde{f}_{n,0}) \ {\rm d}\vq\right]\dx
\]
and note that
\[
\int_\Omega \frac{\rho_{n,0}^{\gamma_n}}{\gamma_n - 1} = \int_\Omega \frac{\rho_0^n}{n - 1} \leq |\Omega|
\]
since $0 \leq \rho_0 \leq 1$, so it follows that $E_{n,0}\leq E(0) + |\Omega|$. Therefore,
\[
\muns\left|\frac{\nabla_x \vu_n + \nabla_x^T\vu_n}{2} - \frac{1}{3}(\Div_x \vu_n)\mathbb{I}\right|^2,  \ \ \munb|\Div_x\vu_n|^2 \text{ uniformly bounded in } L^1(0, T; L^1(\Omega))
\]
due to the energy inequality. Defining
\[
g(\vu_n, \mu_n^S) = \sqrt{\muns} \left(\frac{\nabla_x \vu_n + \nabla_x^T\vu_n}{2} - \frac{1}{3}(\Div_x \vu_n)\mathbb{I}\right)
\]
and
\[
h(\vu_n, \mu_n^B) = \sqrt{\munb} (\Div_x\vu_n),
\]
it follows that
\[
g(\vu_n, \mu_n^S) \rightharpoonup \overline{g(\vu_n, \mu_n^S)} \text{ weakly in }L^2(0, T; L^2(\Omega))
\]
and
\[
h(\vu_n, \mu_n^B) \rightharpoonup \overline{h(\vu_n, \mu_n^B)} \text{ weakly in }L^2(0, T; L^2(\Omega)).
\]
Due to the strong convergence of $\muns, \munb$ we have
\[
\overline{g(\vu_n, \mu_n^S)} = g(\vu, \mu^S) \text{ and }\overline{h(\vu_n, \mu^B_n)} = h(\vu, \mu^B),
\]
so from Tonelli's weak lower semicontinuity theorem it follows that
\[
\int_0^t \int_\Omega h^2(\vu, \mu^S) + g^2(\vu, \mu^B) \leq \liminf_{n \to \infty} \int_\Omega h^2(\vu_n, \mu_n^S) + g^2(\vu_n, \mu^B_n).
\]
Next, we use the strong convergence of $f_n$, and thus of $\tilde{f}_n$, along with Fatou's Lemma and the fact that $\mathcal{F}$ is nonnegative to deduce that
\[
\|\mathcal{F}(\tilde{f}_n)\|_{L^\infty(0,  T; L_M^1(\Omega \times D))} \leq \liminf_{n \to \infty} \|\mathcal{F}(\tilde{f})\|_{L^\infty(0, T; L_M^1(\Omega \times D))},
\]  
up to a subsequence. Combining these results with the convergence results in \eqref{eq:conv} and further applications of Tonelli's theorem for weak lower semicontinuity, and the choice of initial data yield
\begin{align*}
&\int_\Omega \left[\frac{1}{2}\rho|\vu|^2  + \delta \eta^2 + k \int_D M \mathcal{F}(\tilde{f}) \ {\rm d}\vq\right](t, \cdot ) \dx\\
& \quad + \int_0^t\int_\Omega \left(\mus\left|\frac{\nabla_x \vu + \nabla_x^T\vu}{2} - \frac{1}{3}(\Div_x \vu)\mathbb{I}\right|^2 + \mub|\Div_x\vu|^2 \right)(t', \cdot)\dxdt' \\
& \quad + 2\epsilon\delta \int_0^T\int_\Omega \left(|\nabla_x \eta|^2 + \epsilon k  \int_D M |\nabla_x \sqrt{\tilde{f}}|^2 {\rm d} \vq \right)(t', \cdot) \dxdt'\\
& \quad + \frac{kA_0}{4\lambda}\int_0^t\int_\Omega \int_D M \left|\nabla_{\vq}\sqrt{\tilde{f}}\right|^2 \ {\rm d}\vq \ \dxdt'\\
& \leq\int_\Omega \left[\frac{1}{2}\rho_0|\vu_0|^2  + \delta \eta_0^2 + k \int_D M \mathcal{F}(\tilde{f}_0) \ {\rm d}\vq\right](t, \cdot ) \dx +  \liminf_{n \to \infty}\int_\Omega \frac{\rho_{n,0}^{\gamma_n}}{\gamma_n - 1} \dx.
\end{align*}
Then, we note that
\[
\liminf_{n \to \infty}\int_\Omega \frac{\rho_{n,0}^{\gamma_n}}{\gamma_n - 1} \dx = \liminf_{n \to \infty}\int_\Omega \frac{\rho_0^n}{n - 1} \dx = 0,
\]
since $0 \leq \rho_0 \leq 1$. Therefore, $(\rho, \vu, \pi, f, \eta)$ satisfies the energy inequality \eqref{eq:FBenergy} and is a weak solution to the problem ($\vc{P_F}$) in the sense of Definition \ref{def:FBsol}. This concludes the proof of Theorem \ref{thm:FBexist}.

\subsection{Proof of Theorem \ref{thm:stab}}\label{sec:stabproof}

We now set out to prove Theorem \ref{thm:stab}. In Section \ref{sec:unifbounds} we will determine a priori bounds on the quantities of interest by using the assumptions on the initial data, along with the energy inequality \ref{eq:energy}, and in Section \ref{sec:L1} we prove a uniform $L^1$ bound on the quantity $\rho_n^{\gamma_n}$. These uniform bounds will lead to convergence results established in Section \ref{sec:conv}. In Section \ref{sec:conv} we will also prove that the limiting solution is in fact a solution to the problem ($\vc{P_F}$), in the sense of Definition \ref{def:FBsol}, by verifying that the constraint $0 \leq \rho \leq 1$ and the free boundary condition \eqref{eq:FBC4} are satisfied a.e. in $(0, T)\times \Omega$, the divergence-free condition \eqref{eq:FBC1} is satisfied a.e. in $\{\rho = 1\}$. Throughout the rest of the paper, all convergence results are up to a subsequence.

\subsubsection{A Priori Estimates}\label{sec:unifbounds}
The uniform boundedness and weak convergence assumptions for the initial data imply that 
\begin{equation*}
E_{n,0}  \leq C
\end{equation*}
uniformly in $n$. Following the procedure outlined in \cite{FLS}, the following bounds are uniform in $n$:
\begin{equation} \label{eq:unifbounds}
\begin{gathered}
\rho_n |\vu_n|^2 \in L^\infty(0, T; L^1(\Omega)), \ \ \rho_n \in L^\infty(0, T; L^1(\Omega)) \\
\rho_n \vu_n \in L^\infty(0, T; L^{\frac{2\gamma_n}{\gamma_n + 1}}(\Omega)),\\
\eta_n \in 
L^2(0, T; \dot{H}^1(\Omega)), \\
\vu_n \in L^2(0, T; W_0^{1,2}(\Omega)) \text{ when }\omega \geq 0,\\
\vu_n \in 
L^{\frac{4}{2 + |\omega|}}(0, T; W_0^{1, \frac{12}{6 + |\omega|}}(\Omega)) \text{ when }\omega \leq 0\\
\mathbb{S}_n \in L^{\frac{20}{10 + 3\omega}}((0, T)\times \Omega) \text{ when }\omega \geq 0,\\
\mathbb{S}_n \in L^2((0, T)\times \Omega) \text{ when }\omega \leq 0,\\
\mathbb{T}_n \in L^2(0, T; L^{\frac{4}{3}}(\Omega)),\\
\mathcal{F}(\tilde{f}_n) \in L^\infty(0, T; L_M^1(\Omega \times D)) .
\end{gathered}
\end{equation}
Unfortunately, from the energy inequality we have only the bound
\[
\int_\Omega \rho_n^{\gamma_n} \dx \leq c(\gamma_n - 1),
\]
which is not uniform in $n$. Therefore, the next issue is to prove a uniform bound in $n$ for $\rho_n$.
It follows from the energy inequality and the initial condition \eqref{eq:IC3} that
\[
\int_\Omega \rho_n^{\gamma_n} \dx \leq (\gamma_n - 1)E_{n,0} + \int_\Omega (\rho_{n,0})^{\gamma_n}\dx \leq (\gamma_n - 1)E_{n,0} + c\gamma_n \le c \gamma_n,
\]
where the constant $c$ is independent of $n$. Fix $p \in (1, \infty)$. For sufficiently large $n$ we have $\gamma_n > p$, and from the Holder inequality
\[
\|\rho_n\|_{L^\infty(0, T; L^p(\Omega))} \leq \|\rho_n\|_{L^\infty(0, T; L^1(\Omega))}^{\theta_n} \|\rho_n\|_{L^\infty(0, T; L^{\gamma_n}(\Omega))}^{1 - \theta_n} \lesssim V_n^{\theta_n}(c\gamma_n)^{\frac{1 - \theta_n}{\gamma_n}},
\]
where $\frac{1}{p} = \theta_n + \frac{1 - \theta_n}{\gamma_n}$. Recalling that $\gamma_n \to \infty, \ V_n \to V$ as $n \to \infty$, it follows that $\theta_n \to \frac{1}{p}$ as $n \to \infty$, and
\[
\lim_{n \to \infty} V_n^{\theta_n} = V^{1/p}, \ \ \lim_{n \to \infty}(c\gamma_n)^{\frac{1 - \theta_n}{\gamma_n}} = 1.
\]
Therefore,
\[
\lim_{n \to \infty} V_n^{\theta_n}(c\gamma_n)^{\frac{1 - \theta_n}{\gamma_n}} = V^{1/p},
\]
and for sufficiently large $n$,
\[
\|\rho_n\|_{L^\infty(0, T; L^p(\Omega))} \lesssim V^{1/p}
\]
independently of $n$. Thus, $\rho_n$ is uniformly bounded in $L^\infty(0, T; L^p(\Omega))$ for all $1 \leq p < \infty$. Moreover, we have
\begin{equation}\label{eq:linfty}
\sup_n\|\rho_n\|_{L^\infty(0, T; L^{\gamma_n}(\Omega))} \leq \sup_n (c\gamma_n)^{\frac{1}{\gamma_n}} \lesssim 1,
\end{equation}
by the same argument as above.

\subsubsection{$L^1$ regularity for $\rho_n^{\gamma_n}$}\label{sec:L1}
The previous estimates \eqref{eq:linfty} give us a uniform bound on $\rho_n$ in $L^\infty(0, T; L^{\gamma_n}(\Omega))$. However, in order to demonstrate that the sequence $\{\rho_n^{\gamma_n}\}$ will converge in the space of measures $\mathcal{M}((0, T)\times \Omega)$, we need to prove a uniform bound in $L^1((0, T)\times \Omega)$ on the quantity $\rho_n^{\gamma_n}$. 

First, we assume that $\rho_n^{\gamma_n + 1}$ is uniformly bounded in $L^1((0, T)\times \Omega)$. Then, 
\begin{align*}
\int_0^T \int_\Omega \rho_n^{\gamma_n} \dx&= \int_0^T \left(\int_{\Omega \cap \{\rho_n \leq 1\}} \rho_n^{\gamma_n} \ \dx + \int_{\Omega \cap \{\rho_n > 1\}} \rho_n^{\gamma_n} \ \dx\right)\\
& \leq \int_0^T \int_\Omega (\rho_n + \rho_n^{\gamma_n + 1}) \dx.
\end{align*}
Since $\rho_n \in L^\infty(0, T; L^1(\Omega))$ from \eqref{eq:unifbounds}, it follows that 
\[
\rho_n^{\gamma_n} \text{ uniformly bounded in } L^1((0, T)\times \Omega).
\]

We now prove that  $\rho_n^{\gamma_n + 1}$ is uniformly bounded in $L^1((0, T)\times \Omega)$, following the method introduced by Fereisl \cite{Feireisl2001}. We define the test function
\[
\varphi_n(t, x):= \phi(t) \mathcal{B} \left(S_\epsilon [b(\rho_n)] - \frac{1}{|\Omega|} \int_\Omega S_\epsilon [b(\rho_n)]\right),
\]
where $\phi \in C_c^\infty([0, T])$ is a nonnegative test function, $S_\epsilon$ is the classical mollifier in the spatial variable, and $b \in C^1([0, \infty))$ is a function such that $sb'(s) \approx b(s)$, and $|sb'(s)| + |b(s)| \leq c < \infty$ for all $s \in [0, \infty)$. Then we have
\begin{equation}\label{eq:2}
\partial_t S_\epsilon[b(\rho_n)] + \Div_x(S_\epsilon[b(\rho_n)]\vu_n) +S_\epsilon \left[(b'(\rho_n)\rho_n - b(\rho_n)) \Div_x \vu_n\right] = r_{n,\epsilon},
\end{equation}
where $\lim_{\epsilon \to 0} r_{n,\epsilon} = 0$, as shown in Lemma 2.1 in \cite{Feireisl2001}.

Taking $\varphi_n$ as a test function in the $n$th momentum equation \eqref{eq:mom}, and utilizing \eqref{eq:2}, yields
\begin{align*}
&\int_0^T \int_\Omega \phi \rho_n^{\gamma_n} S_\epsilon [b(\rho_n)] \ \dxdt \\
&\quad= \int_0^T \int_\Omega \phi \rho_n^{\gamma_n} \left(\frac{1}{|\Omega|} \int_\Omega S_\epsilon[b(\rho_n)] \ {\rm d}y\right) \ \dxdt\\
&\quad - \int_0^T \int_\Omega \partial_t \phi \rho_n \vu_n \cdot \mathcal{B} \left(S_\epsilon[b(\rho_n)] - \frac{1}{|\Omega|}\int_\Omega S_\epsilon[b(\rho_n)] \ {\rm d}y\right)\dxdt\\
&\quad+ \int_0^T \int_\Omega \phi \rho_n \vu_n \cdot \mathcal{B}\bigg(S_\epsilon [(b'(\rho_n)\rho_n - b(\rho_n))\Div_x \vu_n]\\
&\quad \qquad - \left.\frac{1}{|\Omega|}\int_\Omega S_\epsilon[(b'(\rho_n)\rho_n - b(\rho_n))\Div_x \vu_n] {\rm d}y\right)\dxdt\\
&\quad- \int_0^T \int_\Omega \phi \rho_n \vu_n \cdot \mathcal{B}\left(r_{n,\epsilon}  - \frac{1}{|\Omega|}\int_\Omega r_{n,\epsilon} {\rm d}y\right) \ \dxdt\\
&\quad + \int_0^T \int_\Omega \phi \rho_n \vu_n \cdot \mathcal{B} \left(\Div_x(S_\epsilon[b(\rho_n)]\vu_n)\right) \dxdt\\
&\quad - \int_0^T \int_\Omega \phi \rho_n \vu_{n,i}\vu_{n,j}\partial_{x_i} \mathcal{B}_j \left(S_\epsilon[b(\rho_n)] - \frac{1}{|\Omega|}\int_\Omega S_\epsilon[b(\rho_n)] \ {\rm d}y\right)\dxdt\\
&\quad + \int_0^T \int_\Omega \phi \mathbb{S}_n:\nabla_x \mathcal{B} \left(S_\epsilon[b(\rho_n)] - \frac{1}{|\Omega|}\int_\Omega S_\epsilon[b(\rho_n)] \ {\rm d}y\right)\dxdt\\
&\quad - \int_0^T \int_\Omega \phi\eta_n^2 \left(S_\epsilon[b(\rho_n)] - \frac{1}{|\Omega|}\int_\Omega S_\epsilon[b(\rho_n)] \ {\rm d}y\right)\dxdt\\
&\quad + \int_0^T \int_\Omega \phi \mathbb{T}_{i,j}(f_n)\partial_{x_i}\mathcal{B}_j\left(S_\epsilon[b(\rho_n)] - \frac{1}{|\Omega|}\int_\Omega S_\epsilon[b(\rho_n)] \ {\rm d}y\right)\dxdt\\
&\quad= \sum_{i = 1}^{9} I_i.
\end{align*}
We then use the properties of the operator $\mathcal{B}$ along with the previously proven a priori estimates to bound each term individually. Since the regularity of the fluid velocity $\vu$ is dependent on the value of $\omega$, we consider two cases: the case when $\omega \geq 0$ and the case when $\omega \leq 0$.
\newline
\newline
\textit{Case: $0 \leq \omega < \frac{5}{3}$.} For $I_1$ we have 
\begin{align*}
I_1 & \leq \int_0^T \int_\Omega \phi (\rho_n^{\gamma_n + 1} + \rho_n) \left(\frac{1}{|\Omega|} \int_\Omega S_\epsilon[b(\rho_n)] \ {\rm d}y\right) \ \dxdt\\
& \leq \int_0^T \int_\Omega \phi \rho_n^{\gamma_n + 1} \left(\frac{1}{|\Omega|} \int_\Omega S_\epsilon[b(\rho_n)] \ {\rm d}y\right) \ \dxdt \\
& \quad \quad+C(T, \Omega)\|\rho_n\|_{L^\infty(0, T; L^{\gamma_n}(\Omega))}\|b(\rho_n)\|_{L^\infty(0, T; L^{\gamma_n}(\Omega))} .
\end{align*}
For $I_2$, we have
\begin{align*}
I_2  
& \lesssim  \|\rho_n\vu_n\|_{L^\infty(0, T; L^{\frac{2\gamma_n}{\gamma_n + 1}}(\Omega))}\|b(\rho_n)\|_{L^\infty(0, T; L^{\frac{6\gamma_n}{5\gamma_n - 3}}(\Omega))}
\end{align*}
For $I_3$ we have
\begin{align*}
I_3 
& \lesssim  \|\rho_n\|_{L^\infty(0, T; L^{\gamma_n}(\Omega))} \|\nabla_x \vu_n\|_{L^2(0, T; L^2(\Omega))}^2\|b(\rho_n)\|_{L^\infty(0, T; L^{\frac{3\gamma_n}{2\gamma_n - 3}}(\Omega))}.
\end{align*}
Next, for $I_4$ we have 
\begin{align*}
I_4 & \lesssim \|\rho_n \vu_n\|_{L^\infty(0, T; L^{\frac{2\gamma_n}{\gamma_n + 1}}(\Omega))}\|r_{n,\epsilon}\|_{L^\infty(0, T; L^{\frac{6\gamma_n}{5\gamma_n - 3}}(\Omega))}.
\end{align*}
For $I_5 + I_6$ we have
\begin{align*}
I_5 + I_6 
& \lesssim \|\rho_n\|_{L^\infty(0, T; L^{\gamma_n}(\Omega))}\| \vu_n\|_{L^2(0, T; L^{6}(\Omega))}^2 \|b(\rho_n)\|_{L^\infty(0, T; L^{\frac{3\gamma_n}{2\gamma_n - 3}}(\Omega))}.
\end{align*}

For $I_7 $ we have
\begin{align*}
I_7 & \lesssim \|\mathbb{S}_n\|_{L^{\frac{20}{10 +3\omega}}((0, T)\times \Omega)}\|b(\rho_n)\|_{L^\infty(0, T; L^{\frac{20}{10 - 3\omega}}(\Omega))},
\end{align*} 
which is satisfied when $\omega < \frac{10}{3}$. For $I_8$ we have
\begin{align*}
I_8 & \lesssim \|\eta_n\|_{L^2(0, T; L^6(\Omega))}^2 \|b(\rho_n)\|_{L^\infty(0, T; L^{\frac{3}{2}}(\Omega))},
\end{align*}
and for $I_9$ we have
\begin{align*}
I_9 \lesssim \|\mathbb{T}_n\|_{L^2(0, T; L^{\frac{4}{3}}(\Omega))} \|b(\rho_n)\|_{L^\infty(0, T; L^4(\Omega))}.
\end{align*}
Thus, due to the uniform bounds \eqref{eq:unifbounds} we have
\begin{align*}
\int_0^T \int_\Omega \phi \rho_n^{\gamma_n} S_\epsilon[b(\rho_n)] \ \dxdt & \lesssim_T \int_0^T\int_\Omega \phi \rho_n^{\gamma_n + 1} \left(\frac{1}{|\Omega|} \int_\Omega S_\epsilon[b(\rho_n)] {\rm d}y\right)\dxdt \\
&  + \|b(\rho_n)\|_{L^\infty(0, T; L^{\gamma_n}(\Omega))}^2 + \|b(\rho_n)\|_{L^\infty(0, T; L^{\frac{6\gamma_n}{5\gamma_n - 3}}(\Omega))}\\
&  + \|b(\rho_n)\|_{L^\infty(0, T; L^{\frac{3\gamma_n}{2\gamma_n - 3}}(\Omega))}+ \|r_{n,\epsilon}\|_{L^{\frac{6\gamma_n}{5\gamma_n - 3}}((0, T)\times\Omega)}\\
& +  \|b(\rho_n)\|_{L^\infty(0, T; L^{\frac{20}{10 - 3\omega}}(\Omega))} + \|b(\rho_n)\|_{L^\infty(0, T;L^{\frac{3}{2}}(\Omega))} \\
& + \|b(\rho_n)\|_{L^\infty(0, T; L^4(\Omega))}.
\end{align*}
We then take $\epsilon \to 0$, letting $b(\rho_n)$ approximate $\rho_n$ and $\phi$ approximate 1. This yields
\begin{align*}
\int_0^T \int_\Omega \rho_n^{\gamma_n + 1} \dxdt & \lesssim \int_0^T\int_\Omega \rho_n^{\gamma_n + 1} \left(\frac{1}{|\Omega|} \int_\Omega \rho_n{\rm d}y\right)\dxdt \\
& + \|\rho_n\|_{L^\infty(0, T; L^{\gamma_n}(\Omega))}^2 + \|\rho_n\|_{L^\infty(0, T; L^{\frac{6\gamma_n}{5\gamma_n - 3}}(\Omega))} + \|\rho_n\|_{L^\infty(0, T; L^{\frac{3\gamma_n}{2\gamma_n - 3}}(\Omega))}\\
& + \|\rho_n\|_{L^\infty(0, T; L^{\frac{20}{10 - 3\omega}}(\Omega))} + \|\rho_n\|_{L^\infty(0, T;L^{\frac{3}{2}}(\Omega))} + \|\rho_n\|_{L^\infty(0, T; L^4(\Omega))}.
\end{align*}
Noting that
\[
\frac{1}{|\Omega|} \int_\Omega \rho_n{\rm d}y  \leq V_n< V < 1,
\]
and $\rho_n \in L^\infty(0, T; L^{\gamma_n}(\Omega))$ uniformly in $n$, it follows that
\begin{align*}
\int_0^T \int_\Omega  \rho_n^{\gamma_n + 1} \dxdt & \lesssim 1 +  \|\rho_n\|_{L^\infty(0, T; L^{\frac{6\gamma_n}{5\gamma_n - 3}}(\Omega))} + \|\rho_n\|_{L^\infty(0, T; L^{\frac{3\gamma_n}{2\gamma_n - 3}}(\Omega))}\\
&+ \|\rho_n\|_{L^\infty(0, T; L^{\frac{20}{10 - 3\omega}}(\Omega))} + \|\rho_n\|_{L^\infty(0, T;L^{\frac{3}{2}}(\Omega))} + \|\rho_n\|_{L^\infty(0, T; L^4(\Omega))}.
\end{align*}
Recalling the constraint $0 \leq \omega < \frac{5}{3}$, from \eqref{eq:gammaomega}, it follows that $\rho_n^{\gamma_n + 1}$ is uniformly bounded in $L^1((0, T)\times \Omega)$, and thus 
\[
\rho_n^{\gamma_n} \text{ uniformly bounded in } L^1((0, T)\times \Omega)\text{, provided } \gamma_n \geq 4.
\] 
%
\newline
\newline
\textit{Case: $-\frac{4}{3} < \omega \leq 0$.}  Due to the decreased regularity of $\vu_n$ we must treat the terms $I_3, I_5, I_6, I_7$ differently from the previous case. For $I_3$ we have
\begin{align*}
I_3 
& \lesssim  \|\rho_n\vu_n\|_{L^\infty(0, T; L^{\frac{2\gamma_n}{\gamma_n + 1}}(\Omega))}\|\nabla_x \vu_n\|_{L^{\frac{4}{2 - \omega}}(0, T; L^{\frac{12}{6 -\omega}}(\Omega))}\|b(\rho_n)\|_{L^\infty(0, T; L^{\frac{12\gamma_n}{(4 + \omega)\gamma_n - 6}}(\Omega))},
\end{align*}
and for $I_5, I_6$ we have
\begin{align*}
I_5  + I_6
& \lesssim \|\rho_n \vu_n\|_{L^\infty(0, T; L^{\frac{2\gamma_n}{\gamma_n + 1}}(\Omega))}\|\nabla_x \vu_n\|_{L^{\frac{4}{2 - \omega}}(0, T; L^{\frac{12}{6 - \omega}}(\Omega))} \|b(\rho_n)\|_{L^\infty(0, T; L^{\frac{12\gamma_n}{(4 + \omega)\gamma_n - 6}}(\Omega))}.
\end{align*}
These inequalities are valid under the conditions \eqref{eq:gammaomega} on $\omega$ provided
\[
\gamma_n > \frac{6}{4 + \omega}.
\]
Since $\omega > -\frac{4}{3}$, this condition is satisfied if $\gamma_n > \frac{9}{4}$. Finally, for $I_7$ we have
\begin{align*}
I_7  & \lesssim \|\mathbb{S}_n\|_{L^2((0, T)\times \Omega)}\|b(\rho_n)\|_{L^\infty(0, T; L^2(\Omega))}.
\end{align*} 
Thus, in the case that $-\frac{4}{3} < \omega \leq 0$, we obtain
\begin{align*}
\int_0^T \int_\Omega  \rho_n^{\gamma_n + 1} \dxdt & \lesssim 1 +  \|\rho_n\|_{L^\infty(0, T; L^{\frac{6\gamma_n}{5\gamma_n - 3}}(\Omega))} + \|\rho_n\|_{L^\infty(0, T; L^{\frac{12\gamma_n}{(4 + \omega)\gamma_n - 6}}(\Omega))}\\
& + \|\rho_n\|_{L^\infty(0, T; L^{2}(\Omega))} + \|\rho_n\|_{L^\infty(0, T;L^{\frac{3}{2}}(\Omega))} + \|\rho_n\|_{L^\infty(0, T; L^4(\Omega))},
\end{align*}
which implies that $\rho_n^{\gamma_n + 1}$ when $\gamma_n \geq 4$, and thus 
\[
\rho_n^{\gamma_n} \text{ is uniformly bounded in } L^1((0, T)\times \Omega) \text{, provided }\gamma_n \geq 4.
\]

\subsubsection{Convergence of Approximating Solutions} \label{sec:conv}
The $L^1$ uniform bound on $\rho_n^{\gamma_n}$ obtained in the previous section implies that
\begin{equation*}
\rho_n^{\gamma_n} \rightharpoonup \pi \text{ in }\mathcal{M}((0, T)\times \Omega).
\end{equation*}
Following the procedure outlined in \cite{FLS}, we obtain the following convergence results:
\begin{equation*}
\begin{gathered}
\vu_n \rightharpoonup \vu \text{ in }L^2(0, T; W_0^{1,2}(\Omega)) \text{ when }\omega \geq 0,\\
 \vu_n \rightharpoonup \vu \text{ in }L^2(0, T; W_0^{1,\frac{4}{2 + |\omega|}}(\Omega)) \cap L^{\frac{4}{2 + |\omega|}}(0, T; W_0^{1, \frac{12}{6 + |\omega|}}(\Omega)) \text{ when }\omega \leq 0,\\
\eta_n \to \eta \text{ in }C_w([0, T]; L^2(\Omega)), \\
\eta_n \rightharpoonup \eta \text{ in }L^2(0, T; W^{1,2}(\Omega)),\\
f_n \to f \text{ in }L^1(0, T; L^1(\Omega \times D)),\ \ f \in C_w([0, T]; L^1(\Omega \times D)),\\
\nabla_{\vq}\sqrt{\tilde{f}_n} \rightharpoonup \nabla_{\vq}\sqrt{\tilde{f}}, \ \ \nabla_x \sqrt{\tilde{f}_n} \rightharpoonup\nabla_x \sqrt{\tilde{f}} \text{ in }L^2(0, T; L_M^2(\Omega \times D)),\\
(\mu^S_n, \mu^B_n) \to (\mu^S, \mu^B) \text{ in }L^{\frac{10}{3\omega}}((0, T)\times \Omega) \text{ when }\omega \geq 0,\\
(\mu^S_n, \mu^B_n) \to (\mu^S, \mu^B) \text{ in }L^q((0, T) \times \Omega) \text{ for any }q < \infty \text{ when }\omega \leq 0,\\
\mathbb{S}_n \rightharpoonup \mathbb{S} \text{ in }L^{\frac{10}{3\omega + 5}}((0, T)\times \Omega) \text{ when }\omega \geq 0,\\
\mathbb{S}_n \rightharpoonup \mathbb{S} \text{ in }L^r((0, T)\times \Omega) \text{ for any } r < \frac{20}{3\omega + 10} \text{ when }\omega \geq 0,\\
\mathbb{T}_n \to \mathbb{T} \text{ in }L^r((0, T) \times \Omega) \text{ for any }r < \frac{20}{13}.
\end{gathered}
\end{equation*}
It can also be shown, as in \cite{FLS}, that the nonlinear terms in the Fokker-Planck equation \eqref{eq:FP} converge in the sense of distributions. 

For terms involving the fluid density $\rho_n$, we follow the same procedure as in \cite{LionsMasmoudi-1999}, using the convergence estimates presented in \cite{FLS} to determine that
\begin{equation*}
\begin{gathered}
\rho_n \to \rho \text{ in }C_w([0, T]; L^p(\Omega)) \text{ for } 1 \leq p < \infty,\\
\rho_n \vu_n \to \rho \vu \text{ in }C_w(0, T; L^r(\Omega)) \text{ for } 1 \leq r < 2,\\
\rho_n \vu_n \otimes \vu_n \rightharpoonup \rho \vu \otimes \vu \text{ in } L^2(0, T; L^r(\Omega)) \text{ for some } r > 1.
\end{gathered}
\end{equation*}
These convergence results are sufficient to demonstrate that $(\rho, \vu, \pi, f, \eta)$ solves the equations \eqref{eq:FBcont}-\eqref{eq:FBdiff} in the sense of distributions. It remains to show that the free boundary conditions are satisfied and that the density satisfies $0 \leq \rho \leq 1$ a.e. in $(0, T)\times  \Omega$.

\subsubsection{Convergence of $(\rho_n - 1)_+ \to 0$}\label{sec:rholim}
Since $\rho_n \geq 0$ a.e., it follows that $\rho \geq 0$ a.e. as well. Then, we need to demonstrate that $(\rho_n - 1)_+ \to 0$. First, we define $\phi_n$ by
\[
\phi_n = (\rho_n - 1)_+,
\]
and note that
\[
\int_\Omega (1 + \phi_n)^{\gamma_n} \mathbbm{1}_{\{\phi_n > 0\}} \leq \int_\Omega \rho_n^{\gamma_n} \leq c\gamma_n.
\]
We recall the inequality
\[
(1 + x)^k \geq 1 + c_p k^p x^p \text{ for } k \text{ sufficiently large, } p > 1, \ x > 0,
\]
and take $k = \gamma_n,$ $x = \phi_n$ to obtain
\[
 c_p \gamma_n^p \int_\Omega\phi_n^p \leq |\Omega| c_p \gamma_n^p \int_\Omega \phi_n^p \leq \int_\Omega (1 + \phi_n)^{\gamma_n} \leq c \gamma_n.
\]
This yields
\[
\int_\Omega \phi_n^p \leq \frac{c}{c_p \gamma_n^{p - 1}}.
\]
Taking $n \to \infty$, we have
\[
(\rho_n - 1)_+ \to 0 \text{ in }L^\infty(0, T; L^p(\Omega)) \text{ for all } 1 \leq p < \infty.
\]
Thus, $0 \leq \rho \leq 1$ pointwise a.e.

\subsubsection{Free Boundary Conditions} \label{sec:FBC} 
We now set out to prove that the free boundary conditions \eqref{eq:FBC1}-\eqref{eq:FBC3} are satisfied. First, we note that the conditions \eqref{eq:FBC2} and \eqref{eq:FBC3} are equivalent to the single condition 
\begin{equation} \label{eq:FBC4}
\rho \pi = \pi \geq 0 \text{ a.e. in }(0, T)\times \Omega,
\end{equation}
since the pressure vanishes in the region $\{\rho <1\}$. Then, we set
\begin{equation*}
s_n= \rho_n \log \rho_n, \ \ s = \rho \log \rho.
\end{equation*}
Then, the continuity equation \eqref{eq:cont} yields
\begin{equation}\label{eq:5}
\partial_t s_n + \Div_x(s_n \vu_n) + (\Div_x \vu_n) \rho_n = 0.
\end{equation}
Applying the operator $(-\Delta)^{-1}\Div_x $ to the momentum equation \eqref{eq:mom} yields
\begin{equation}\label{eq:4}
\partial_t (\mathcal{A}_i[\rho_n \vu_n^i]) - \mathcal{R}_{ij}[\rho_n \vu_n^i \vu_n^j] + \rho_n^{\gamma_n} + \mathcal{R}_{ij}[\mathbb{S}_n^{ij}] = - \mathcal{R}_{ij}[\mathbb{T}_n^{ij}],
\end{equation}
where we have utilized the Riesz-type operators $\mathcal{A}, \mathcal{R}$. From the properties of the Riesz operator $\mathcal{R}$, it follows that
\[
\sum_{i,j = 1}^3 \mathcal{R}_{ij}[\mathbb{S}_n^{ij}] =\left (\munb - \frac{1}{3} \muns\right) \Div_x \vu_n + \sum_{i,j = 1}^3 \left( \muns \mathcal{R}_{ij}[\partial_j \vu_n^i] + R_n^{ij}\right)
\]
in the sense of distributions, where $R_n^{ij} = \mathcal{R}_{ij}[\muns \partial_j \vu_n^i] - \muns\mathcal{R}_{ij}[\partial_j \vu_n^i]$ is a commutator. The symmetry of the Riesz operator implies that
\[
\sum_{i,j = 1}^3 \mathcal{R}_{ij}[\partial_j \vu_n^i]  = \Div_x \vu_n,
\]
so \eqref{eq:4} becomes
\begin{align*}
\begin{split}
&\partial_t (\mathcal{A}_i[\rho_n \vu_n^i]) - \mathcal{R}_{ij}[\rho_n \vu_n^i \vu_n^j] + \rho_n^{\gamma_n} + \left(\frac{2}{3}\muns + \munb\right) \Div_x \vu \\
& \quad = -R^{ij}_n - \mathcal{R}_{ij}[\mathbb{T}_n^{ij}].
\end{split}
\end{align*}
Multiplying by $\rho_n$ and comparing to \eqref{eq:5} yields
\begin{align*}
\begin{split}
&\left(\munb + \frac{2}{3} \muns\right)\left[\partial_t s_n + \Div_x (s_n \vu_n)\right] + \rho_n^{\gamma_n + 1} \\
& \quad = \rho_n \mathcal{R}_{ij}[\mathbb{T}_n^{ij}] + \partial_t \left(\rho_n\mathcal{A}_{i}[\rho_n \vu_n^i]\right) + \Div_x \left(\rho_n \vu_n\mathcal{A}_i[\rho_n \vu_n^i]\right)\\
& \quad\quad + \rho_nQ_n^{ij} + \rho_n R_n^{ij},
\end{split}
\end{align*}
where $Q_n^{ij}$ is the commutator
\[
Q_n^{ij} = \vu_n^i \mathcal{R}_{ij}[\rho_n \vu_n^j] - \mathcal{R}_{ij}[\rho_n \vu_n^i \vu_n^j].
\]
Taking the limit as $n \to \infty$ yields
\begin{align*}
\begin{split}
&\left(\mub + \frac{2}{3} \mus\right)\left(\partial_t \overline{s} + \Div_x (\overline{s} \vu)\right) + \overline{\rho^{\gamma + 1}} \\
& \quad = \rho\mathcal{R}_{ij}[\mathbb{T}^{ij}]  + \partial_t \left(\rho\mathcal{A}_i[\rho\vu^i]\right) + \Div_x \left(\rho \vu\mathcal{A}_i[\rho \vu^i]\right)\\
& \quad\quad + \rho Q^{ij} + \rho R^{ij}.
\end{split}
\end{align*}
Here, we have used the convergence results stated in the previous section, along with uniform bounds on $\partial_t \rho_n$, $\partial_t(\rho_n \vu_n)$ and the following compensated compactness lemma (Lemma 3.3 in \cite{LionsMasmoudi-1999}).
\begin{lemma}\label{lem:comp}
Suppose that $\{g_n\} \subset L^{p_1}(0, T; L^{p_2}(\Omega))$, $\{h_n\} \subset L^{q_1}(0, T; L^{q_2}(\Omega))$ are two sequences such that $g_n \rightharpoonup g$ in $L^{p_1}(0, T; L^{p_2}(\Omega))$ and $h_n \rightharpoonup h$ in $L^{q_1}(0, T; L^{q_2}(\Omega))$, where $1 \leq p_1, p_2 \leq \infty$, and $\frac{1}{p_1} + \frac{1}{q_1} = \frac{1}{p_2} + \frac{1}{q_2} = 1$. In addition, assume that
\[
\partial_t g_n \text{ is uniformly bounded in } \mathcal{M}(0, T; W^{-m,1}(\Omega)) \text{ for some } m \geq 0,
\]
and
\[
h_n \text{ is uniformly bounded in }L^1(0, T; H^s(\Omega)) \text{ for some }s > 0.
\]
Then, $g_n h_n \rightharpoonup gh$ in the sense of distributions.
\end{lemma}
Additionally, the convergence of $\rho_n R_n^{ij} \rightharpoonup\rho R^{ij}$ and $\rho_n Q_n^{ij} \rightharpoonup\rho Q^{ij}$ is obtained by following the procedure in \cite{FLS}, \cite{LionsMasmoudi-1999}. Specifically, we use the properties of the Riesz operator $\mathcal{R}$ to demonstrate that the commutators $Q_n^{ij}, R_n^{ij}$ are uniformly bounded in some space $L^1(0, T; W^{s, p}(\Omega))$, with $s > 0, p> 1$ and use Lemma \ref{lem:comp} (or a variant in the case that $p < 2$) to prove weak convergence of the products $\rho_n R_n^{ij}, \rho_n Q_n^{ij}$. 

The same procedure applied to the limiting equations \eqref{eq:FBcont} and \eqref{eq:FBmom} yields
\begin{align*}
\begin{split}
&\left(\mub + \frac{2}{3} \mus\right)\left(\partial_t s + \Div_x (s \vu)\right) + \rho \pi \\
& \quad = \rho\mathcal{R}_{ij}[\mathbb{T}^{ij}]  + \partial_t \left(\rho\mathcal{A}_i[\rho\vu^i]\right) + \Div_x \left(\rho \vu\mathcal{A}_i[\rho \vu^i]\right)\\
& \quad\quad + \rho Q^{ij} + \rho R^{ij}.
\end{split}
\end{align*}
Comparison of the equations for $s$ and for $\overline{s}$ yields
\begin{equation*}
\left(\mub + \frac{2}{3} \mus\right)\left(\partial_t (\overline{s} - s) + \Div_x ((\overline{s} - s)\vu)\right) = \rho \pi -  \overline{\rho^{\gamma + 1}}.
\end{equation*}
Due to the growth constraints on $\mu^S$ and $\mu^B$ we have $(\mub + \frac{2}{3} \mus)^{-1}$ integrable, so we can write
\begin{equation} \label{eq:6}
\partial_t (\overline{s} - s) + \Div_x ((\overline{s} - s)\vu)=\left(\mub + \frac{2}{3} \mus\right)^{-1} \left( \rho \pi -  \overline{\rho^{\gamma + 1}}\right).
\end{equation}
Noting that
\[
\rho^{\gamma_n} \to \mathbbm{1}_{\{\rho = 1\}} \text{ in }L^p((0, T)\times \Omega) \text{ for all }1 \leq p < \infty,
\]
it follows that 
\[
\rho^{\gamma_n}(\rho_n - \rho) \rightharpoonup 0,
\]
and thus
\begin{equation*}
\rho \pi - \overline{\rho^{\gamma + 1}} = \overline{ \left((\rho - \rho_n) \rho_n^{\gamma_n}\right)} = \overline{ \left((\rho - \rho_n)(\rho_n^{\gamma_n} - \rho^{\gamma_n})\right)}  \geq 0.
\end{equation*}
Integrating \eqref{eq:6} over $\Omega$ yields
\begin{equation} \label{eq:7}
\partial_t \int_\Omega (\overline{s} - s) = \int_\Omega \left(\mub + \frac{2}{3} \mus\right)^{-1} \left( \rho \pi -  \overline{\rho^{\gamma + 1}}\right) \leq 0.
\end{equation}
Since $s$ is concave function, we must have $s \leq \overline{s}$. Additionally, due to the strong convergence assumption $\rho_{n,0} \to \rho_0$ in $L^1(\Omega)$, it follows that $(\overline{s} - s)|_{t = 0} = 0$, and thus  \eqref{eq:7} implies that $\overline{s} \leq s$, and $\overline{s} = s$ at all times $t$. From \eqref{eq:6} it follows that $\rho \pi = \overline{\rho^{\gamma + 1}}$, since $(\mu^B + \frac{2}{3}\mu^S) > 0$.

The equality $\overline{s} = s$, along with the strong convergence of $\rho_{n,0} \to \rho_0$ in $L^1(\Omega)$, gives a pointwise a.e. convergence result for $\rho_n$, since $s$ is a convex function. Thus, $\rho_n \to \rho$ in $L^p((0, T)\times \Omega)$ for all $1 \leq p < \infty$. Following the strategies outlined in \cite{LionsMasmoudi-1999}, we can also obtain the strong convergence results
\begin{gather*}
\rho_n \vu_n \to \rho \vu \text{ in } L^p(0, T; L^r(\Omega)) \text{ for all }1 \leq p < \infty, 1 \leq r < 2,\\
\rho_n \vu_n \otimes \vu_n \to \rho \vu \otimes \vu \text{ in }L^p(0, T; L^1(\Omega)) \text{ for all } 1 \leq p < \infty. 
\end{gather*}
Next, we fix $\epsilon > 0$. For sufficiently large $n$, it follows that
\[
\rho_n^{\gamma_n + 1} \geq \rho_n^{\gamma_n} - \epsilon,
\]
and taking the weak limit gives
\[
\rho \pi = \overline{\rho_n^{\gamma_n + 1}} \geq \overline{\rho_n^{\gamma_n}} - \epsilon = \pi - \epsilon.
\]
Formally, since $\rho \leq 1$ it follows that $\rho \pi \leq \pi$. However, since $\pi \in \mathcal{M}$ the product $\rho \pi$ is not defined a.e., so we have to make sense of the inequality $\rho \pi \leq \pi$. To address this issue we use mollifiers to define sequences of smooth approximating functions $\pi_\epsilon, \rho_\epsilon$. We can write
\[
(\rho - 1)\pi = (\rho_\epsilon - 1)\pi_\epsilon + (\rho - \rho_\epsilon)\pi_\epsilon + (\rho - 1 )(\pi - \pi_\epsilon)
\]
Taking $\epsilon \to 0$ and using $0 \leq \rho \leq 1$ yields $(\rho - 1)\pi \leq 0$, and thus $\rho \pi = \pi$. For more details, see \cite{LionsMasmoudi-1999}.

It remains to show that the incompressibility condition $\Div_x \vu = 0$ is satisfied a.e. in $\{\rho  = 1\}$, 
which is a result of the following lemma (Lemma 2.1 in \cite{LionsMasmoudi-1999}). 

\begin{lemma}
Let $\vu \in L^2(0, T; H^1_{\text{loc}}(\Omega))$ and $\rho \in L_{\text{loc}}^2((0, T)\times \Omega)$ such that $\partial_t \rho + \Div_x(\rho \vu) = 0$ in $(0, T)\times \Omega$, in the sense of distributions, and $\rho(0) = \rho_0$. Then, the following assertions are equivalent.
\begin{itemize}
\item [(i)] $\Div_x \vu = 0$ a.e. on $\{\rho \geq 1\}$ and $0 \leq \rho_0 \leq 1$.
\item [(ii)] $0 \leq \rho \leq 1$. 
\end{itemize}
\end{lemma}
We note that this lemma only strictly applies in the case when $0 \leq \omega < \frac{5}{3}$, due to a loss of regularity for $\vu$ when $-\frac{4}{3} < \omega \leq 0$. However, it can be demonstrated that the lemma still applies in the case when $\vu \in L^2(0, T; W^{1, \frac{4}{2 + |\omega|}}(\Omega))$, for example, since $\frac{4}{2 + |\omega|} > 1$. This concludes the proof of Theorem \ref{thm:FBexist}.

\section{Weak Sequential Stability}\label{SWSS}
In this section we present the weak sequential stability of the family of dissipative (finite-energy) weak solutions to the free boundary problem $(\vc{P_F}).$
The results is presented below.

\begin{theorem}[Weak sequential stability]\label{WSS}
Let $\{ (\rho_n, \vu_n, \psi_n, \eta_n)\}_{n \in \mathbb{N}}$ be a sequence of dissipative (finite energy) weak solutions in the sense of Definition \ref{def:FBsol} associated with the initial data $\{(\rho_{0,n}, \vu_{0,n}, \psi_{0,n}, \eta_{0,n})\}_{n \in \mathbb{N}}$ satisfying:
\begin{enumerate}[\quad$\star$]
\item 
$\rho_{0,n} \ge 0$ a.e. in $\Omega$, $\rho_{0,n} \to \rho_0$ strongly in $L^{1}(\Omega)$;
\smallskip

\item
$\vu_{0,n} \to \vu_0$ in $L^r(\Omega; \mathbb{R}^3)$ for some $r>1$ such that $\rho_{0,n}|u_{0,n}|^2 \to \rho_0 |u_0|^2$ strongly in $L^1(\Omega)$;
\smallskip

\item
$\psi_{0,n} \ge 0$ a.e. in $\Omega \times D, \,\,\, \psi_{0,n} \to \psi_0, \,\,   \psi_{0,n}\left(\log \frac{\psi_{0,n}}{M}\right) \to \psi_0 \left(\log \frac{\psi_{0,n}}{M}\right)$ strongly in $L^1(\Omega \times D);$
\smallskip

\item
$\eta_{0,n} =\int_D \psi_{0,n} dq \to \eta_0$ strongly in $L^2(\Omega).$
\end{enumerate}
Let $\vc{f} \in L^{\infty}((0,T)\times \Omega; \mathbb{R}^3).$ Suppose that the parameter $\omega$ in \eqref{eq:omegacond} satisfy 
\begin{equation}
-\frac{4}{3}<  \omega< \frac{5}{3},
\end{equation}
Then, there exists a subsequence  such that 
$(\rho_n, \vu_n, \psi_n, \eta_n) \to (\rho, \vu, \psi, \eta),\,\, \mbox{as}\,\, n\to \infty,$ in the sense of distributions, where the limit $(\rho, \vu, \psi, \eta)$ is a dissipative
finite energy solution in the sense of Definition \ref{def:FBsol} associated with the initial data $(\rho_0, \vu_0, \psi_0, \eta_0).$
\end{theorem}
\begin{proof}
Consider a sequence of the dissipative  (finite energy) weak solutions satisfying the assumptions in Theorem \ref{WSS} as well as the energy inequality 
 \eqref{eq:FBenergy}.
The result is obtained by following the line of argument of Theorem  \ref{thm:stab} in Section \ref{sec:proof}, with several exceptions:
\begin{enumerate}[\quad$\star$]
\item Higher integrability of the fluid pressure is established by using the bound $0 \leq \rho_n \leq 1$, which yields a uniform bound on $\rho_n$ in $L^p((0, T)\times \Omega)$ for any $1 \leq p \leq \infty$. Thus, $\rho_n \rightharpoonup \rho$ in $L^p((0, T)\times \Omega)$.
\smallskip
\item The multipliers method employed in the proof of Theorem \ref{thm:stab} is adapted to demonstrate that the $L^1$ bound on $\pi_n$ is controlled by the initial data, thus obtaining a uniform bound on $\pi_n$. This demonstrates that $\pi_n \rightharpoonup \pi$ in $\mathcal{M}((0, T)\times \Omega)$.
\smallskip
\item The strong convergence of the initial data is used to demonstrate that the limiting solutions satisfies the energy inequality \eqref{eq:FBenergy}.
\end{enumerate}
\end{proof}


\begin{thebibliography}{00}
\bibitem{BT2012} H. Bae and K. Trivisa, On the Doi model for the suspensions of rod-like molecules in compressible fluids, {\it Math. Models Methods Appl. Sci.}  {\bf 22} (2012),  39 pp.
\bibitem{BT2013} H. Bae and K. Trivisa, On the Doi model for the suspensions of rod-like molecules: Global-in-time existence,  {\it Commun.\ Math.\ Sci.} {\bf 11} (2013), 831-850.
\bibitem{BS2011} J. W. Barrett and E. S\"{u}li, Existence and equilibration of global weak solutions to kinetic models for dilute polymers I: Finitely extensible nonlinear bead-spring chains, {\it Math. Models Methods Appl. Sci.} {\bf 21} (2011), 1211-1289.

\bibitem{BS2012A} J. W. Barrett and E. S\"{u}li, Existence and equilibration of global weak solutions to kinetic models for dilute polymers II: Hookean-type bead-spring chains, {\it Math. Models Methods Appl. Sci.} {\bf 22} (2012), 84 pp.

\bibitem{BS2012B} J. W. Barrett and E. S\"{u}li, Existence of global weak solutions to finitely extensible nonlinear bead-spring chain models for dilute polymers with variable density and viscosity, {\it J. Differential Equations} {\bf 253} (2012), 3610-3677.



\bibitem{BS2016} J. W. Barrett and E. S\"{u}li, Existence of global weak solutions to compressible isentropic finitely extensible nonlinear bead-spring chain models for dilute polymers, {\it  Math. Models Methods Appl. Sci.} {\bf 26} (2016), {\bf no. 3}, 469-568.










\bibitem{DT2017} D. Donatelli and K. Trivisa, On a free boundary problem for polymeric fluids: global existence of weak solutions, {\it NoDEA Nonlinear Differential Equations Appl.} {\bf 24} (2017), {\bf no. 5}, Art. 51, 20 pp.

\bibitem{DT2018} D. Donatelli and K. Trivisa, On a free boundary problem for finitely extensible bead-spring chain molecules in dilute polymers, {\it J. Math. Anal. Appl.}  {\bf 482} (2020), {\bf no. 1}, 24 pp.


\bibitem{Feireisl2001} E. Feireisl, On compactness of solutions to the compressible isentropic Navier-Stokes equations when the density is not square integrable, {\it Comment. Math. Univ. Carolin.}, {\bf 42} (2001), {\bf no.1}, 83--98.


\bibitem{FLS} E. Feireisl, Y. Lu, and E. S\"{u}li, Dissipative weak solutions to compressible Navier-Stokes-Fokker-Planck systems with variable viscosity coefficients, {\it J. Math. Anal. App.}, {\bf 443} (2016), {\bf no. 1}.


\bibitem{LionsMasmoudi-1999} P.L.\ Lions and N.\ Masmoudi, On a free boundary barotropic model, {\it Ann. \ Inst. Henri Poincar\'{e}}, {\bf 16} (1999) {no. 3}, 373-410.



\end{thebibliography}
\end{document}